\documentclass[a4paper]{amsart}
\usepackage{amssymb}

\input xy
\xyoption{all}

\newcommand{\Id}{\mathrm{Id}}
\newcommand{\Comp}{\mathsf{Comp}}

\newcommand{\Cl}{\mathrm{Cl}}

\newcommand{\R}{\mathbb R}

\newcommand{\F}{\mathcal F}
\newcommand{\C}{\mathcal C}
\newcommand{\M}{\mathbb M}
\newcommand{\E}{\mathbb E}

\newtheorem{theorem}{Theorem}
\newtheorem{df}{Definition}
\newtheorem{lemma}{Lemma}

\newtheorem{problem}{Problem}
\newtheorem{proposition}{Proposition}

\begin{document}

\title{On exact  capacities}


\author{Taras Radul}

\maketitle

Institute of Mathematics, Casimirus the Great University of Bydgoszcz, Poland;
\newline
Department of Mechanics and Mathematics, Ivan Franko National University of Lviv,
Universytettska st., 1. 79000 Lviv, Ukraine.
\newline
e-mail: tarasradul@yahoo.co.uk


\begin{abstract} We consider  capacity (fuzzy measure, non-additive probability) on a compactum as a monotone  cooperative normed game. We introduce topological analogue of well known class of exact  games and show that this class forms a subfunctor of the capacity functor which lie between known subfunctors of convex capacities and balanced capacities. It is natural to consider  probability measures as elements of core of such games. We describe exact capacities as envelopes of the convex closed sets of  probability measures.  We also consider strongly exact capacities and pose the problem of coincidence of these two classes.
\end{abstract}

\maketitle

\section{Introduction}
Capacities (non-additive measures, fuzzy measures) were introduced by Choquet in \cite{Ch} as a natural generalization of additive measures. They found numerous applications (see for example \cite{EK},\cite{Gil},\cite{Sch}). Capacities on compacta were considered in \cite{Lin} where the important role plays the upper-semicontinuity property which  connects the capacity theory with the topological structure. Categorical and topological properties of spaces of upper-semicontinuous normed capacities on compact Hausdorff spaces  were investigated in \cite{NZ}. In particular, there was built the capacity functor which is a functorial part of a capacity monad $\M$.

Let us remark that each monad structure describes some algebraic properties of a functor and leads to some abstract convexity (see for example \cite{R1}, \cite{R2} and \cite{R8}). Convexity structures are widely used to prove existence of fixed points and equilibria (see for example \cite{E}, \cite{BCh}, \cite{KZ}, \cite{Rad}, \cite{R3}, \cite{R4}, \cite{R6}, \cite{R7} ). The class of all capacities contains two important subclasses of possibility and necessity capacities which form submonads of the capacity monad and, hence, these subspaces inherit corresponding convexity structures which was used in   \cite{Rad}, \cite{R3}, \cite{R4} in order to prove existence of equilibria. However another important subclass of convex capacities does not form a submonad of the capacity monad \cite{NZ}.

Capacities can be considered as monotone cooperative games.  Important subclasses of cooperative games are convex, exact and balanced games (see for instance \cite{Grab} for more details). Convex capacities on compacta were considered in \cite{NZ}, balanced capacities on compacta were introduced in  \cite{R5}, both construction form subfunctors of the capacity functor.

We introduce exact  capacities on compacta in this paper and show that these notions lead to some subfunctor of the capacity functor.

\section{Exact and  totally balanced capacities} By compactum we mean a compact Hausdorff space.  In what follows, all spaces are assumed to be compacta except for $\R$ and maps are assumed to be continuous.  By $\F(X)$ we denote the family of all closed subsets of $X$.

For each compactum $X$ we denote by $C(X)$ the Banach space of all
continuous functions $\phi:X\to\R$ with the usual $\sup$-norm: $
\|\phi\| =\sup\{|\phi(x)|\mid x\in X\}$. We also consider on $C(X)$ the natural partial order. For $\lambda\in\R$ we denote by $\lambda_X$ the constant function on $X$ with values $\lambda$. For a subset $D\subset \R$ we denote by $C(X,D)$ the subset of $C(X)$ consisting of functions with the codomain $D$.

We need the definition of capacity on a compactum $X$. We follow a terminology of \cite{NZ}.
A function $\nu:\F(X)\to [0,1]$  is called an {\it upper-semicontinuous capacity} on $X$ if the three following properties hold for each closed subsets $F$ and $G$ of $X$:
\begin{enumerate}
\item $\nu(X)=1$, $\nu(\emptyset)=0$;
\item if $F\subset G$, then $\nu(F)\le \nu(G)$;
\item if $\nu(F)<a$, then there exists an open set $O\supset F$ such that $\nu(B)<a$ for each compactum $B\subset O$.
\end{enumerate}

If $F$ is a one-point set we use a simpler notation $\nu(a)$ instead $\nu(\{a\})$.
A capacity $\nu$ is extended in \cite{NZ} to all open subsets $U\subset X$ by the formula $\nu(U)=\sup\{\nu(K)\mid K$ is a closed subset of $X$ such that $K\subset U\}$.

It was proved in \cite{NZ} that the space $MX$ of all upper-semicontinuous  capacities on a compactum $X$ is a compactum as well, if a topology on $MX$ is defined by a subbase that consists of all sets of the form $O_-(F,a)=\{c\in MX\mid c(F)<a\}$, where $F$ is a closed subset of $X$, $a\in [0,1]$, and $O_+(U,a)=\{c\in MX\mid c(U)>a\}$, where $U$ is an open subset of $X$, $a\in [0,1]$. Since all capacities we consider here are upper-semicontinuous, in the following we call elements of $MX$ simply capacities.

The notion of convex capacity (under another name) appeared in \cite{Ch}.  A capacity $\nu\in MX$ is called convex if for all closed subsets $A$, $B\subset X$ of a compactum $X$ we have $\nu(A\cup B) + \nu(A \cap B) \ge \nu(A) + \nu(B)$. We denote by $MCX$ the subspace of the space $MX$ consisting of convex capacities.  It was shown in \cite{NZ} that $MCX$ is closed in $MX$.

A capacity $\nu\in MX$ is called balanced if we have $\sum_{i=1}^n\lambda_i\nu(A_i)\le 1$ for each numbers $\lambda_1,\dots,\lambda_n\in\R$ and closed subsets $A_1,\dots,A_n$ of $X$ such that $\sum_{i=1}^n\lambda_i\chi_{A_i}\le 1_X$ where $\chi_{A_i}$ is the characteristic function of the set $A_i$ \cite{R5}. We denote by $MBX$ the subspace of the space $MX$ consisting of balanced capacities.  It was shown in \cite{R5} that $MBX$ is closed in $MX$.

The main aim of this section is to introduce topological analogues of important class of  exact capacities (or more generally games) well known in the frames of finite $X$.

By probability measure on a compactum $X$ we mean normed $\sigma$-additive regular Borel measure. We denote by $PX$ the set of all probability measure on a compactum $X$. Evidently we have $PX\subset MX$, we consider the subspace topology on $PX$.


For a compactum $X$ by $\exp X$ we denote the set of non-void
compact subsets of $X$ provided with the Vietoris topology. A subbase
of this topology consists of the sets of the form
$$<U>=\{A\in \exp\ X\mid A\subset U\}$$
and $$<X,U>=\{A\in \exp\ X\mid \ A\cap U\ne\emptyset$$  where $U$ is  open in $X$. The space $\exp\ X$ is called the
hyperspace of $X$.

For any convex compactum $X$, $ccX$ is defined to be the set of all nonempty closed convex
subsets of $X$, the space $ccX$ is considered as the subspace of the hyperspace $expX$. For any affine mapping $f : X\to Y$
the map $ccf:ccX\to ccY$  is given by $ccf(A) = f(A)$ where $A \in ccX$. The functor $cc$ was studied
in \cite{BZR}

Consider the map $sX:cc(PX)\to MX$ defined as follows $sX(\alpha)(A)=\inf\{\mu(A)\mid\mu\in \alpha\}$ for $\alpha\in cc(PX)$ and $A\in\F(X)$.

\begin{proposition}\label{scont} The map $sX$ is well-defined and continuous.
\end{proposition}

\begin{proof} It is easy to check that $sX$ is well-defined, i.e. $sX(cc(PX))\subset MX$. Consider any $\alpha\in cc(PX)$ such that $sX(\alpha)\in O_+(U,a)$ for some $a\in[0,1]$ and open $u\subset X$. Then we have $\inf\{\mu(A)\mid\mu\in \alpha\}=b>a$. Put $c=\frac{a+b}{2}$. We have $\alpha\in<O_+(U,c)>$ and $$sX(<O_+(U,c)>)\subset O_+(U,a).$$

Now, consider any $\alpha\in cc(PX)$ such that $sX(\alpha)\in O_-(A,a)$ for some $a\in[0,1]$ and closed $A\subset X$. Then we have  $\alpha\in<PX,O_-(A,a)>$ and $$sX(<PX,O_-(A,a)>)\subset O_-(A,a).$$
\end{proof}

The following definition is a topological analogue for capacities of the definition of exact game given in \cite{Sch3}.

\begin{df}\label{exact} A capacity $\nu\in MX$ is called exact if  $\nu\in sX(cc(PX))$.
\end{df}

We denote by $MEX$ the subspace of the space $MX$ consisting of exact capacities.

\begin{proposition}\label{closedE} The set $MEX$ is closed and convex in $MX$.
\end{proposition}

\begin{proof} Proposition \ref{scont} implies that  $MEX$ is closed. Take any $\nu,\xi\in MEX$ and $\lambda\in[0,1]$. There exist $\alpha$, $\beta\in cc(PX)$ such that $sX(\alpha)=\nu$ and $sX(\beta)=\xi$. Consider the set $$(1-\lambda)\alpha+\lambda\beta=\{(1-\lambda)\mu_1+\lambda\mu_2\mid \mu_1\in\alpha\ \mu_2\in\beta\}\in cc(PX).$$
It is easy to see that $sX((1-\lambda)\alpha+\lambda\beta)=(1-\lambda)\nu+\lambda\xi$, hence $(1-\lambda)\nu+\lambda\xi\in MEX$.
\end{proof}

\begin{df}\label{core}\cite{R5}  Let $\nu\in MX$ be a capacity. We say that a probability measure $\mu$ belongs to the core of $\nu$ if $\mu(A)\ge\nu(A)$ for each closed subset $A$ of $X$.
\end{df}

For $\nu\in MX$ we denote by $core\nu$ the set of all probability measures $\mu$ belonging to the core of $\nu$. It was proved in \cite{R5} that $core\nu\ne\emptyset$ if and only if $\nu\in MBX$.

\begin{lemma}\label{coreCl} For each $\nu\in MX$ $core\nu$ is a closed convex subset of $PX$.
\end{lemma}

\begin{proof} It is easy to see that $core\nu$ is convex. Consider any $\mu\notin core\nu$. Then there exists a closed subset  $A\subset X$ such that $\mu(A)<\nu(A)$. Put $O=\{\xi\in PX\mid\xi(A)<\nu(A)$. Then $O$ is an open neighborhood of $\mu$ in $PX$ which is disjoint with $core\nu$.
\end{proof}

The following definition is another topological analogue of the definition of exact game given in \cite{Sch3}.

\begin{df}\label{sexact} A capacity $\nu\in MX$ is called strongly exact if for each closed subset $B\subset X$ there exists $\mu\in core\nu$ such that $\nu(A)=\mu(A)$.
\end{df}

Let us remark that $\inf$  in the definition of the map $sX$ becomes $\min$  for any finite compactum $X$, hence the notions of exactness and strong exactness coincide. However  when $X$  is infinite the infimum not always  is  attained as we can see in the following example.

Consider the compact space $X =[0,1] $ and its closed subset  $A=\{0\}$. We define a sequence of probability measures as follows: $\mu_n=\frac{1}{n}\delta_0+\frac{n-1}{n}\delta_{\frac{1}{n}}$ ($\delta_t$ is the Dirac measure concentrated in the point $t\in[0,1]$). We have evidently $\mu_n\rightarrow\mu_1=\delta_0$.  Put $\alpha=\Cl(conv\{\mu_n|n\in\mathbb {N}\}))$. Then $\nu(A)=0$, but $\mu(A)>0$ for each $A\in\alpha$.

The coincidence problem is open in general case. We can formulate it in terms of probability measures.

\begin{problem} Let $\alpha\in cc(PX)$ and $A$ is a closed subset of $X$. If there exists $\nu\in PX$ such that $\nu(A)=\inf\{\mu(A)\mid\mu\in \alpha\}$ and $\nu(B)\ge\inf\{\mu(B)\mid\mu\in \alpha\}$ for each closed $B\subset X$.
\end{problem}

For finite $X$ we have the following chain of inclusions   $$MCX\subset MEX\subset  MBX$$ (see for example \cite{Grab}).   The inclusion $MCX\subset MBX$ was proved in \cite{R5} for all compacta. Let us remark that for each $\nu\in MSEX$ we have $\nu= sX(core\nu)$, hence $MSEX\subset MEX$ for each compactum $X$. The inclusion $MEX\subset MBX$ is trivial.

\begin{theorem} Let $X$ be a compactum and  $\nu$ is a convex capacity on $X$. Then $\nu$ is  strongly exact.
\end{theorem}

\section{Categorical properties of balanced capacities}

By $\Comp$ we denote the category of compact Hausdorff
spaces (compacta) and continuous maps.
 For a continuous map of compacta $f:X\to Y$ we define the map $Mf:MX\to MY$ by the formula $Mf(\nu)(A)=\nu(f^{-1}(A))$ where $\nu\in MX$ and $A\in\F(Y)$. The map $Mf$ is continuous.  In fact, this extension of the construction $M$ defines the capacity functor $M$ in the category $\Comp$ (see \cite{NZ} for more details). It was shown in  \cite{NZ} and \cite{R5} respectively that the constructions $MC$ and $MB$ form the subfunctors of $M$. It means that $Mf(MCX)\subset MCY$ and $Mf(MBX)\subset MBY$.

\begin{lemma}\label{nat} Let $f:X\to Y$ be a continuous map between compacta $X$ and $Y$. Then we have  $Mf\circ sX=sY\circ ccPf$.
\end{lemma}

\begin{proof} Take any $\alpha\in ccPX$ and $A\in\F(X)$. Then we have $$Mf(sX(\alpha))(A)=sX(\alpha)(f^{-1}(A))=\inf\{\nu(f^{-1}(A))\mid\nu\in\alpha\}=$$ $$=\inf\{Pf(\nu)(A)\mid\nu\in\alpha\}=\inf\{\mu)(A)\mid\mu\in ccPf(\alpha)\}=sY(ccPf(\alpha))(A).$$ \end{proof}

Let us remark that Lemma \ref{nat} means in the language of the category theory that the family $s=\{sX\mid X\in\Comp$ is a natural transformation of the functor $ccP$ into the functor $M$.

\begin{lemma}\label{sub} Let $f:X\to Y$ be a continuous map between compacta $X$ and $Y$. Then we have  $Mf(MEX)\subset MEY$.
\end{lemma}

\begin{proof} Take any $\nu\in Mf(MEX)$. Then there exists $\mu\in MEX$ such that $Mf(\mu)=\nu$. By the definition of exact capacity there exists $\alpha\in ccPX$ such that $sX(\alpha)=\mu$. We have $\nu=Mf(sX(\alpha))=sY(ccPf(\alpha))$ by Lemma\ref{nat}. Since $ccPf(\alpha)\in ccPY$, we  have $\nu\in MEY$.
\end{proof}

Proposition  \ref{closedE} and Lemma \ref{sub} imply that  $ME$ is a subfunctor of $M$.

We recall the notion  of monad (or triple) in the sense of S.Eilenberg and J.Moore \cite{EM}.  We define it only for the category $\Comp$.

A {\it monad} \cite{EM} $\E=(E,h,m)$ in the category $\Comp$ consists of an endofunctor $E:{\Comp}\to{\Comp}$ and natural transformations $h:\Id_{\Comp}\to F$ (unity), $m:F^2\to F$ (multiplication) satisfying the relations $$m\circ Eh=m\circ h E=\text{\bf 1}_E$$ and $$m\circ m E=m\circ Em.$$ (By $\Id_{\Comp}$ we denote the identity functor on the category ${\Comp}$ and $E^2$ is the superposition $E\circ E$ of $E$.)

The functor $M$ was completed to the monad $\M=(M,h,m)$ in \cite{NZ}, where the components of the  natural transformations are defined as follows:
$$
h X(x)(F)=\begin{cases}
1,&x\in F,\\
0,&x\notin F;\end{cases}
$$

$$m X(\C)(F)=\sup\{t\in[0,1]\mid \C(\{c\in MX\mid c(F)\ge t\})\ge t\},$$ where $x\in X$, $F$ is a closed subset of $X$ and $\C\in M^2X$.

However, it was shown  \cite{NZ} and \cite{R5} respectively that neither $MC$ nor $MB$  form a submonad of $\M$.
We will adopt the example from \cite{NZ} in order to show that $ME$   neither forms a  submonad of $\M$. Define  $\C\in P^2\{1,2,3,4\}$ by the formula $$\C=0.2\delta_{0.1\delta_1+0.1\delta_2+0.8\delta_4}+0.2\delta_{0.1\delta_1+0.1\delta_3+0.8\delta_4}+
0.2\delta_{0.1\delta_2+0.1\delta_3+0.8\delta_4}+
0.4\delta_{\delta_4}.$$

Since $core\mu=\mu$ for each $\mu\in PX$, we have  the inclusion $PX\subset MEX$, hence $P^2X\subset ME^2X$ for each compactum $X$.

\begin{proposition} We have  $m \{1,2,3,4\}(\C)\notin ME\{1,2,3,4\}$.
\end{proposition}

\begin{proof} We will show more: $m \{1,2,3,4\}(\C)\notin MB\{1,2,3,4\}$. Put $\nu=m \{1,2,3,4\}(\C)$.
Suppose the contrary, $\mu\in core\nu$.  It was shown in \cite{NZ}  that $\nu(\{4\})=0.8$, $\nu(\{2,3,4\})=0.9$, $\nu(\{1\})=0.1$, $\nu(\{1,2\})=0.2$ and $\nu(\{1,3\})=0.2$. Then we have $\mu(\{1\})\ge0.1$ and $\mu(\{2,3,4\})\ge0.9$. The additivity of $\mu$ implies that $\mu(\{1\})=0.1$. Since $\mu(\{1,2\})\ge0.2$, we have $\mu(\{2\})\ge0.1$. Analogously, $\mu(\{3\})\ge0.1$. Hence $\mu(\{1,2,3,4\})\ge 0.1+0.1+0.1+0.8=1.1$ and we obtain a contradiction.
\end{proof}

Let us remark that the previous proposition implies that $ME$ neither   forms a  submonad of $\M$.

\end{document}